\documentclass[10pt,a4paper]{amsart}
\usepackage[english]{babel}
\usepackage{amsmath,amsthm,mathrsfs,hyperref}
\usepackage{amssymb}
\usepackage{aliascnt}
\usepackage{comment}
\usepackage{todonotes}
\usepackage{xspace}
\usepackage{nicefrac}
\DeclareMathOperator{\image}{''}

\DeclareMathOperator{\crit}{crit}

\DeclareMathOperator{\Col}{Col}

\DeclareMathOperator{\power}{\mathcal{P}}

\newcommand{\Todorcevic}{Todor\v{c}evi\'c\xspace}

\newcommand{\GCH}{{\rm GCH}\xspace}

\newtheorem{theorem}{Theorem}
\newaliascnt{example}{theorem}

\aliascntresetthe{example}
\newaliascnt{fact}{theorem}

\aliascntresetthe{fact}
\newaliascnt{corollary}{theorem}
\newtheorem{corollary}[corollary]{Corollary}
\aliascntresetthe{corollary}
\newaliascnt{lemma}{theorem}
\newtheorem{lemma}[lemma]{Lemma}
\aliascntresetthe{lemma}
\newaliascnt{claim}{theorem}
\newtheorem{claim}[lemma]{Claim}
\aliascntresetthe{claim}
\newaliascnt{remark}{theorem}
\newtheorem{remark}[lemma]{Remark}
\aliascntresetthe{remark}

\theoremstyle{definition}
\newaliascnt{definition}{theorem}
\newtheorem{definition}[definition]{Definition}
\aliascntresetthe{definition}

\newtheorem*{theorem*}{Theorem}
\newtheorem*{remark*}{Remark}
\newtheorem*{example*}{Example}
\newtheorem*{lemma*}{Lemma}
\newtheorem*{conjecture}{Conjecture}

\begin{document}
\title{Partial strong compactness and squares}
\author{Yair Hayut}
\email{yair.hayut@mail.huji.ac.il}
\address{School of Mathematical Sciences.
Tel Aviv University.
Tel Aviv 69978,
Israel}
\begin{abstract}
In this paper we analyze the connection between some properties of partially strongly compact cardinals: the completion of filters of certain size and instances of the compactness of $\mathcal{L}_{\kappa,\kappa}$. Using this equivalence we show that if any $\kappa$-complete filter on $\lambda$ can be extended to a $\kappa$-complete ultrafilter and $\lambda^{<\kappa} = \lambda$ then $\square(\mu)$ fails for all regular $\mu\in[\kappa,2^\lambda]$. As an application, we improve the lower bound for the consistency strength of \emph{$\kappa$-compactness}, a case which was explicitly considered by Mitchell.
\end{abstract}
\maketitle
\section{Introduction}
Strongly compact cardinals are one of the most intriguing large cardinals notions. Strongly compact cardinals were defined by Tarski (for a complete historical overwiew see \cite[Chapter 4]{Kanamori2008HigherInfinite}). While being very natural and well studied, some of their basic properties are still quite mysterious.  

Strongly compact cardinals are characterized by many different global principles. When taking local versions of those principles one obtains different, and sometimes non-equivalent, large cardinal axioms. We will use the following definitions as our versions for local strong compactness:
\begin{definition}
Let $\kappa$ be a regular cardinal and let $\lambda$ be a cardinal. 
\begin{itemize}
\item $\kappa$ has the \emph{$\lambda$-filter extension property} if any $\kappa$-complete filter $\mathcal{F}$ on $\lambda$  can be extended to a $\kappa$-complete ultrafilter.
\item \emph{$\mathcal{L}_{\kappa,\kappa}$-compactness property for languages of size $\lambda$} holds if for every language with $\lambda$ many non-logical symbols $\mathcal{L}$ and every collection $\Phi$ of $\mathcal{L}_{\kappa,\kappa}$-sentences in the language $\mathcal{L}$, if every sub-collection $\Phi'\subseteq\Phi$ of size $<\kappa$ has a model then $\Phi$ has a model.
\item A cardinal $\kappa$ is $\lambda$-strongly compact if there is a fine $\kappa$-complete ultrafilter on $P_\kappa \lambda$\footnote{In \cite[Chapter 22]{Kanamori2008HigherInfinite}, the term \emph{$\lambda$-compactness} is used to denote what we call $\lambda$-strong compactness. We prefer the more cumbersome name in order to avoid an inconsistency with the term \emph{$\kappa$-compact}, which refers to a cardinal $\kappa$ that has the $\kappa$-filter extension property.

This terminology also differs from the terminology of Bagaria and Magidor, in the paper \cite{BagariaMagidor2014}, in which it refers to the degree of completeness of the extended ultrafilters}. 
\end{itemize}
\end{definition}

We will say that a collection of formulas $\Phi$ is \emph{$\kappa$-consistent} if any subset of it of size $<\kappa$ has a model. Thus $\mathcal{L}_{\kappa,\kappa}$-compactness is the statement that $\kappa$-consistence implies consistence.

In this paper we would like to give a level-by-level analysis of those properties. Let $\lambda = \lambda^{<\kappa}$ be a cardinal. Then:
\[
\begin{matrix}
2^\lambda\text{-strong compactness} & \implies & \lambda\text{-filter extension} \\ 
\iff \mathcal{L}_{\kappa,\kappa}\text{-compactness for langauges of size }2^\lambda & \implies& \lambda\text{-strong compactness}.
\end{matrix}
\]
The first implication can be found in \cite[Proposition 4.1]{Kanamori2008HigherInfinite} and the last one is \cite[Theorem 22.17]{Kanamori2008HigherInfinite}. The second equivalence is Theorem \ref{thm: equivalent of compactness and filters} in this paper, which is the main technical result of this paper.

Following Gitik, \cite{Gitik2018}, we say that a cardinal $\kappa$ is \emph{$\kappa$-compact} if it has the $\kappa$-filter extension property. This case was explicitly considered by Mitchell in \cite{Mitchell79}, since its characterization uses only measures on $\kappa$. Mitchell asked about the possibility of existence of a cardinal $\kappa$ which is $\kappa$-compact in a model of the form $L[\mathcal{U}]$.  

Using Theorem \ref{thm: equivalent of compactness and filters} and Lemma \ref{lemma: elementary embedding}, we conclude the following:
\begin{theorem*}
If $\kappa$ has the $\kappa$-filter extension property then $\square(\kappa)$ and $\square(\kappa^+)$ fail. In particular, if $\kappa$ is $\kappa$-compact then there is an inner model with class many Woodin cardinals and class many strong cardinals.
\end{theorem*}

We remark that the $\kappa$-compactness of $\kappa$ has a similar affect on properties of cardinals up to $2^\kappa$ as $2^\kappa$-strong compactness (see Corollary \ref{cor: reflections}). Some of those implications cannot be obtained from weaker reflection principles such as stationary reflection. Instances of Rado's conjecture, which follow from $\kappa$-compactness of $\kappa$ as well, can be used to provide some of the combinatorial reflection properties of $\kappa^+$. Yet, some of the consequences of $\kappa$-compactness are unknown to follow from some instances of Rado's conjecture, for example instances of compactness of the chromatic number of graphs or simultaneous reflection of stationary sets.

The results of this paper suggest that $\kappa$-compactness is essentially a property of the cardinal $2^\kappa$ rather than of $\kappa$ itself. Indeed, large cardinal properties which are formulated in terms of ultrafilters on $\kappa$ tends to be weaker and do not entail any compactness properties at the level of $\kappa^{+}$. For example, measurability and even superstrength or $1$-extentibility of a cardinal $\kappa$ are compatible with $\square_\kappa$ (see, for example, \cite{CummingsSchimmerling2002}). Moreover, in \cite{Gitik2018}, Gitik shows that if one restricts the $\kappa$-filter extension property to some classes of natural filters, then the consistency strength drops significantly. The least large cardinal which implies the failure of squares at its successor seems to be \emph{subcompactness} (see \cite{SchimmerlingZeman}), and the results of this paper suggest that $\kappa$-compactness is related to a strong version of subcompactness. 

The paper is organized as follows. In Section \ref{section:equivalence} we prove the equivalence between levels of the filter extension property and levels of $\mathcal{L}_{\kappa,\kappa}$-compactness. We remark that bounded instances of $\mathcal{L}_{\kappa,\kappa}$-compactness are equivalent to the existence of certain elementary embeddings.
From this we conclude that many local reflection phenomena which are derived from strongly compact cardinals can be obtained from the $\kappa$-filter extension property for $\kappa$ and improve the known lower bounds for its consistency strength (see \cite{Mitchell79} and \cite{Gitik2018}). 
 
In Section \ref{section:upper bounds}, we provide some upper bounds for the consistency strength of the $\lambda$-filter extension property.
\section{Equivalence of compactness and filter extension property}\label{section:equivalence}
In this section we will demonstrate the equivalence between the $\lambda$-filter extension property and $\mathcal{L}_{\kappa,\kappa}$-compactness for languages of size $2^\lambda$, for $\lambda = \lambda^{<\kappa}$. 

The following theorem was proved for the case $\kappa = \lambda = \omega$ by Fichtenholz and Kantorovich and for the case $\kappa = \omega \leq \lambda$ by Hausdorff. Hausdorff's proof generalizes to the case of arbitrary $\kappa$, assuming $\lambda = \lambda^{<\kappa}$. 
\begin{lemma}[Fichtenholz-Kantorovich, Hausdorff]\label{lemma:hausdorff}
Let $\kappa \leq \lambda$ be infinite cardinals, $\lambda = \lambda^{<\kappa}$. There is a family $\mathcal{I} \subseteq \power(\lambda)$, $|\mathcal{I}| = 2^\lambda$ such that for every pair of disjoint collections $\mathcal{A}, \mathcal{B} \subseteq \mathcal{I}$, $|\mathcal{A}|, |\mathcal{B}| < \kappa$ and $\mathcal{A} \neq \emptyset$, 
\[\left|\bigcap_{X \in \mathcal{A}} X \setminus \bigcup_{Y \in \mathcal{B}} Y\right| = \lambda.\]
\end{lemma}
\begin{proof}
We include a proof for this theorem for the completeness of the paper.

Clearly, if is sufficient to find an independent set in $\power(J)$ for some $|J| = \lambda$. Let \[J = \{\langle X, Z\rangle \mid X \in P_\kappa\lambda,\, Z \subseteq \power(X)\}\] 
Since $\lambda^{<\kappa} = \lambda$, we can compute: $\lambda \leq |J| \leq \lambda^{<\kappa} \cdot 2^{<\kappa} = \lambda$.

For a set $A\subseteq \lambda$, let $I(A) = \{\langle X, Z\rangle \mid A\cap X \in Z\}$. Let us show that $\mathcal{I} = \{I(A) \mid A \subseteq \lambda\}$ is as required. 

If $A\neq B$ are subsets of $\lambda$ then there is some ordinal $\gamma \in A \triangle B$. So $\langle \{\gamma\}, \{\{\gamma\}\}\rangle \in I(A) \triangle I(B)$ and in particular $I(A) \neq I(B)$. Therefore $|\mathcal{I}| = 2^\lambda$. 

Let $\mathcal{A} = \{I(A_\alpha) \mid \alpha < \rho\}\subseteq \mathcal{I}$, $\mathcal{B} = \{I(B_\beta) \mid \beta < \zeta\}\subseteq \mathcal{I}$, $\rho,\zeta < \kappa$, and $A_\alpha\neq B_\beta$ for all $\alpha, \beta$. 

We need to show that for $D = \bigcap \mathcal{A} \setminus \bigcup \mathcal{B}$, $\left|D\right| = \lambda$ (when if $\mathcal{A} = \emptyset$, we define $\bigcap \mathcal{A} = J$). Indeed, if $X\in P_\kappa\lambda$ is sufficiently large so that \[\{A_\alpha \cap X\mid \alpha < \rho\} \cap \{B_\beta \cap X \mid \beta < \zeta\} = \emptyset\] then $\langle X, \{A_\alpha \cap X\mid \alpha < \rho\}\rangle \in D$. Since there are $\lambda$ many possibilities for $X$, $|D| = \lambda$.\end{proof}
\begin{theorem}\label{thm: equivalent of compactness and filters}
Let $\kappa, \lambda$ be cardinals and assume $\lambda^{{<}\kappa} = \lambda$. $\kappa$ has the $\lambda$-filter extension property if and only if $\mathcal{L}_{\kappa,\kappa}$-compactness for languages of size $2^\lambda$ holds.
\end{theorem}
\begin{proof}
Using Henkin's classical construction we can reduce the problem of compactness of $\mathcal{L}_{\kappa,\kappa}$ into compactness of the propositional logic $\mathcal{L}_{\kappa,1}$ (without quantifiers, but with conjunctions and disjunctions of size $<\kappa$):

Let $\mathcal{L}$ be a language of size $2^{\lambda}$. Add for each $\mathcal{L}_{\kappa,\kappa}$-formula in the language $\mathcal{L}$ of the form $\exists_{\alpha < \rho} x_\alpha \psi(\langle x_\alpha \mid \alpha < \rho\rangle)$ a sequence of $\rho$ many constants, $\langle c^\psi_{\alpha} \mid \alpha < \rho\rangle$. We would like to interpret those constants such that 
\[\exists_{\alpha < \rho} x_\alpha \psi(\langle x_\alpha \mid \alpha < \rho\rangle) \iff \psi(\langle c^\psi_\alpha \mid \alpha < \rho\rangle).\]

By repeating the process $\kappa$ many times, we may assume that any formula that starts with an existential quantifier has corresponding constants in $\mathcal{L}$. Those constants witness the validity of the formula, is case that it is true and are of arbitrary values if it is false.  

Let us introduce an atomic propositional formula $[[\varphi]]$ for every $\mathcal{L}_{\kappa,\kappa}$-formula $\varphi$ over the language $\mathcal{L}$. Let $\Psi$ be the collection of all formulas of the form:
\begin{enumerate}
\item $[[\varphi]]$ is true if $\varphi$ is a tautology.
\item $[[\varphi]] = \neg [[\neg\varphi]]$.
\item For every $\rho < \kappa$ and every $\rho$-sequence of formulas $\langle \varphi_\alpha \mid \alpha < \rho\rangle$, \[[[\bigwedge_{\alpha < \rho} \varphi_\alpha]] = \bigwedge_{\alpha < \rho} [[\varphi_\alpha]].\]
\item For every $\rho < \kappa$ and every formula $\varphi$, \[[[\exists_{\alpha < \rho} x_\alpha \varphi(\langle x_\alpha \mid \alpha < \rho\rangle)]] = [[\varphi(\langle c^\varphi_\alpha \mid \alpha < \rho\rangle)]]\] and for every sequence of terms $\langle t_\alpha \mid \alpha < \rho\rangle$, \[[[\varphi(\langle t_\alpha \mid \alpha < \rho\rangle)]] \rightarrow [[\varphi(\langle c^\varphi_\alpha \mid \alpha < \rho\rangle)]].\] 
\end{enumerate} 

For a given evaluation of all the variables $\{[[\varphi]] \mid \varphi\text{ is a }\mathcal{L}_{\kappa,\kappa}\text{-formula}\}$, we define an equivalence relation between the terms, by $t\sim t'$ if and only if the truth value of the variable $[[t = t']]$ is true. Let us denote by $[t]_\sim$ the equivalence class of the term $t$. For every relation $R$ of arity $n$ in $\mathcal{L}$, we define $R([t_0]_\sim, \dots, [t_{n-1}]_\sim)$ if $[[R(t_0, \dots, t_{n-1})]]$ is true. For every function $F$ of arity $n$ in the language $\mathcal{L}$ we define $F([t_0]_\sim, \dots, [t_{n-1}]_\sim) = [s]_\sim$ if $[[F(t_0,\dots,t_{n-1})=s]]$ is true. 

For any $\kappa$-consistent theory $T$ over the language $\mathcal{L}$ with the logic $\mathcal{L}_{\kappa,\kappa}$, one can translate the problem of the consistency of $T$ into the problem of constructing an assignment which is consistent with the collection $\Phi$ of propositional variables $[[\varphi]]$ which consists of $\{[[\varphi]] \mid \varphi\in T\}$ and $\Psi$. Clearly, $\Phi$ is still $\kappa$-consistent and any consistent assignment for it provides a model for $T$. 

Let us focus from this point in propositional theories. Let us fix a $\kappa$-independent family in $\power(\lambda)$, $\mathcal{I} = \{A_\delta \mid \delta < 2^\lambda\}$. Such a family exists by Lemma \ref{lemma:hausdorff}. Let $\mathcal{B}\subseteq\power(\lambda)$ be the $\kappa$-complete Boolean algebra which is generated by $\mathcal{I}$. The independence of $\mathcal{I}$ is equivalent to the fact that $\mathcal{B}$ is isomorphic to the $\kappa$-complete free Boolean algebra with $2^\lambda$ generataors. 

Let us define an embedding $\iota$ from the Lindenbaum-Tarski algebra of formulas in $\mathcal{L}_{\kappa,1}$ with atoms $\langle a_\gamma \mid \gamma < 2^\lambda\rangle$ into $\mathcal{B}$ by setting $\iota(a_\gamma) = A_\gamma$ (and inductively, $\iota(\neg \varphi) = \lambda \setminus \iota(\varphi)$ and $\iota(\bigwedge_{i < \eta} \varphi_i) = \bigcap_{i < \eta} \iota(\varphi_i)$). 

For a function $s\colon \Gamma \to 2$, $\Gamma \subseteq 2^\lambda$, $|\Gamma| < \kappa$, let us define: 

\[A(s) = \left(\bigcap_{\alpha\in\Gamma,\, s(\alpha) = 1} A_\alpha\right) \setminus \left(\bigcup_{\beta\in\Gamma,\, s(\beta) = 0} A_\beta\right)\]
with the convention that $\bigcap_{A\in \emptyset} A = \lambda$ (similarly to the convention in the proof of Lemma \ref{lemma:hausdorff}). By the independence of the family $\mathcal{I}$, $A(s) \neq \emptyset$ for all $s\in\,{}^\Gamma 2,\,\Gamma\in P_\kappa 2^\lambda$. Let us define the value of $s(\varphi)$ to be the truth value of $\varphi$ after assigning for each variable $a_\gamma$ in $\Gamma$ the truth value $s(\gamma)$. This value is well defined only when $\Gamma$ contains all the indices of the variables that appear in $\varphi$. 

For a formula $\varphi$, let $\Gamma_\varphi$ denote the set of indices of variables that appear in $\varphi$.

\begin{claim}
For every $\Gamma \supseteq \Gamma_\varphi$, $|\Gamma| < \kappa$,
\[\iota(\varphi) = \bigcup \{ A(s) \mid {s\in\,{}^{\Gamma}2,\, s(\varphi) = 1}\}.\]
\end{claim}
\begin{proof}
Let us show first that for every $\Gamma \subseteq 2^\lambda$, $|\Gamma| < \kappa$, the collection \[\{A(s) \mid s \in\,{}^\Gamma2\}\] is a partition of $\lambda$. Indeed, if $s, s'$ are different assignments then there is $\gamma$ such that $s(\gamma) \neq s'(\gamma)$. In particular, $A(s) \subseteq A_\gamma$ and $A(s') \subseteq \lambda \setminus A_\gamma$, and they are disjoint. In order to show that the union of this collection is $\lambda$, let us pick $\delta \in \lambda$. Let $s_\delta\colon \Gamma\to 2$ be defined as $s_\delta(\gamma)$ is $1$ if $\delta\in A_\gamma$ and $0$ otherwise. Clearly, $\delta\in A(s_\delta)$.

Thus, we conclude that for a given formula $\varphi$, and for any $\Gamma\supseteq \Gamma_\varphi$, $|\Gamma| < \kappa$,
\[\bigcup \{ A(s) \mid {s\in\,{}^{\Gamma}2,\, s(\varphi) = 1}\} = \bigcup \{ A(s) \mid {s\in\,{}^{\Gamma_\varphi}2,\, s(\varphi) = 1}\}.\]
By induction on the complexity of the formula. For atomic formula - this is true by the definition of $\iota(a_\gamma)$. For $\varphi = \neg \psi$, and every assignment with domain $\Gamma_\psi = \Gamma_\varphi$, $s(\varphi) = 1 - s(\psi)$. 
\[\begin{matrix}
\iota(\varphi) & = &\lambda \setminus \iota(\psi) \\
& = & \lambda \setminus \big(\bigcup_{s\in\,{}^{\Gamma_\psi} 2, s(\psi) = 1} A(s)\big)\\
& = & \bigcup_{s\in\,{}^{\Gamma_\psi} 2, s(\psi) = 0} A(s)\\
& = & \bigcup_{s\in\,{}^{\Gamma_\varphi} 2, s(\varphi) = 1} A(s)
\end{matrix}
\] 
where the third equation is based on the observation above that the set of $A(s)$, where $s$ ranges over all assignments for $\Gamma$, is a partition of $\lambda$. 

For $\varphi = \bigwedge_{\alpha <\rho} \psi_\alpha$, let $\Gamma_\alpha$ be $\Gamma_{\psi_\alpha}$. Let $\Gamma = \Gamma_{\varphi} = \bigcup_{\alpha < \rho} \Gamma_\alpha$. Clearly, $|\Gamma| < \kappa$. For all $\alpha < \rho$,
\[\begin{matrix}
\iota(\psi_\alpha) &= &\bigcup\{ A(s) \mid {s \in {}^{\Gamma_\alpha}2,\, s(\psi_\alpha) = 0}\}\\
& = &\bigcup\{ A(s) \mid {s \in {}^{\Gamma}2,\, s(\psi_\alpha) = 0}\}\end{matrix}\]
and thus (since the sets $\{A(s)\mid s\in {}^\Gamma2\}$ are pairwise disjoint):
\[\begin{matrix}
\iota(\varphi) & = & \bigcap_{\alpha < \rho} \iota(\psi_\alpha) \\ 
& = & \bigcup_{s \in {}^\Gamma 2, \forall \alpha < \rho,\, s(\psi_\alpha) = 0} A(s) \\
& = & \bigcup_{s \in {}^\Gamma 2, s(\varphi) = 0} A(s)
\end{matrix}
\]

\end{proof}
In particular, if $\varphi$ is consistent (i.e.\ there is an assignment $s$ such that $s(\varphi)=1$) then $\iota(\varphi) \neq \emptyset$ and if $\varphi$ is inconsistent then $\iota(\varphi) = \emptyset$. This implies that if $\varphi, \psi$ are formulas such that $\varphi\leftrightarrow \psi$ is a tautology then $\iota(\varphi) = \iota(\psi)$ and thus $\iota$ is well defined on the Lindenbaum-Tarski Boolean algebra. 

Let $\Phi$ be a collection of $\mathcal{L}_{\kappa,1}$ propositional formulas, with variables $\{a_\delta \mid \delta < 2^{\lambda}\}$, such that any sub-collection of $<\kappa$ formulas from it has a consistent assignment. 
 
Let $\mathcal{F}$ be the $\kappa$-complete filter which is generated by $\iota(\varphi)$, $\varphi\in \Phi$. $\mathcal{F}$ is a proper filter, since any collection of $<\kappa$ many formulas from $\Phi$ is consistent. Since $\iota$ respects Boolean operations of length $<\kappa$, if $\{\varphi_\alpha \mid \alpha < \rho\}\subseteq \Phi$ and $\rho < \kappa$, then 
\[ \bigcap_{\alpha<\rho} \iota(\varphi_\alpha) = \iota\left(\bigwedge_{\alpha<\rho} \varphi_\alpha\right)\]
and since the formula $\bigwedge_{\alpha<\rho} \varphi_\alpha$ is consistent, the intersection is non-empty.

Let $\mathcal{U}\supseteq\mathcal{F}$ be a $\kappa$-complete ultrafilter. Then $\mathcal{U}$ defines an assignment on the variables $\{a_\delta \mid \delta < 2^\lambda\}$: we set $a_\delta$ to be true if and only if $A_\delta\in \mathcal{U}$. Let $S\in{}^{2^\lambda} 2$ be this assignment. Since $\mathcal{U}$ is $\kappa$-complete, by induction on the complexity of the formula, we can see that for every formula $\varphi$, $(S \restriction \Gamma_\varphi)(\varphi)$ is $1$ if and only if $\iota(\varphi)\in\mathcal{U}$.

Let us turn now to the other direction. Let us assume that $\mathcal{L}_{\kappa,\kappa}$-compactness holds for languages of size $2^\lambda$ and $\mathcal{F}$ is a $\kappa$-complete ultrafilter, one can add a constant $a_X$ for every $X \in \power(\lambda)$, and let $U$ be a unary predicate. Let $\Phi$ be the collection of formulas in the language $\mathcal{L}$ over the logic $\mathcal{L}_{\kappa,\omega}$:
\begin{enumerate}
\item $U(a_X)$ for all $X \in \mathcal{F}$.
\item $\big(\bigwedge_{i < \eta} U(a_{X_i}) \big) \rightarrow U(a_{Y})$, for all sequence $\langle X_i \mid i < \eta < \kappa\rangle$ of subsets of $\lambda$, where $Y \supseteq \bigcap X_i$.
\item $U(a_X) \leftrightarrow \neg U(a_{\lambda \setminus X})$.
\end{enumerate}
Clearly, any model of $\Phi$ will define a $\kappa$-complete ultrafilter that extends $\mathcal{F}$.
\end{proof}

The following lemma is a generalization of the characterization of weakly compact cardinals using elementary embeddings, as in Hauser, \cite{Hauser1991}. The proof is a direct modification of the same proof for weakly compact cardinal (see for example \cite[Theorem 4.5]{Kanamori2008HigherInfinite}).
\begin{lemma}\label{lemma: elementary embedding}
Let $\kappa \leq \lambda$ be uncountable cardinals and let us assume that $\mathcal{L}_{\kappa,\kappa}$-compactness holds for languages of size $\lambda^{<\kappa}$. Then, for every transitive model $M$, $|M| = \lambda^{<\kappa}$, $^{{<}\kappa}M\subseteq M$, $\lambda\subseteq M$, there is a transitive model $N$ and an elementary embedding $j\colon M \to N$ such that:
\begin{enumerate}
\item $\crit j = \kappa$, and in particular $j$ is ${<}\kappa$-continuous.
\item There is $s\in N$, $j\image \lambda \subseteq s$, $|s|^N < j(\kappa)$. 
\end{enumerate}
\end{lemma}
\begin{proof}
Let $\mathcal{L}$ be a language with a constant symbol $c_x$ for all $x\in M$, two additional constants $d, s$ and a binary relation $E$.

Let us consider the set of formulas that consists of all $\mathcal{L}_{\kappa,\kappa}$-elementary diagram of $M$ (using the constants $c_x$ and membership as $E$). Let us add the formulas:
\begin{enumerate}
\item For all $\alpha < \kappa$, we add the formula $c_\alpha E d$. 
\item We add the formula $d E c_\kappa$.
\item For all $\alpha < \lambda$, we add the formula $c_\alpha E s$.
\item We add the formula $|s| < c_\kappa$ (interpreted with the standard set theoretical meaning).
\end{enumerate}
Note that any collection of $<\kappa$ many formulas has a model. Namely, take $s$ to be a set of cardinality $<\kappa$ that contains all ordinals $\alpha$ such that the formula ``$c_\alpha \in s$'' appears in the collection, and take $d$ to some arbitrary ordinal below $\kappa$ which is larger than all ordinals smaller than $\kappa$ that were mentioned in the collection.

Thus, $\Phi$ is $\kappa$-consistent and therefore it has a model. The membership relation of this model, $E$, is well founded since $M$ is well founded. Let $N$ be the transitive collapse of the obtained model. Each element of $M$, $x$ has a corresponding constant in the language, $c_x$. Let $j(x)$ be $c_x^N$: the member of $N$ which is evaluated as $c_x$. The embedding $j$ is elementary, since the elementary diagram of $M$ was included in $\Phi$. 

The critical point of $j$ is at least $\kappa$, since for all $\alpha < \kappa$, the assertion ``$x E c_\alpha \implies \bigvee_{\beta < \alpha} x = c_\beta$'' is $\mathcal{L}_{\kappa,\kappa}$ sentence that appears in the elementary diagram of $M$. Thus, in $N$, there is no new ordinal below $\kappa$. But clearly, $d$ is evaluated as a new ordinal below $j(\kappa)$, and thus $\kappa = \crit j$. 
   
The existence of $s$ is clear by construction. The continuity of $j$ follows from the closure of $M$ under sequences of length $<\kappa$. \end{proof}

\begin{corollary}\label{cor: filter extension to elementary embedding}
Let $\kappa \leq \lambda$ be uncountable cardinals, $\lambda^{<\kappa} = \lambda$. $\kappa$ has the $\lambda$-filter extension property if and only if for every transitive model $M$, $2^\lambda \subseteq M$, $|M| =2^\lambda$, there is a transitive model $N$ and an elementary embedding $j\colon M \to N$, such that $\crit j = \kappa$ and there is $s\in N$, $|s|^N < j(\kappa)$, $j\image 2^\lambda \subseteq s$.
\end{corollary}
\begin{proof}
The first direction follows from Theorem \ref{thm: equivalent of compactness and filters} and Lemma \ref{lemma: elementary embedding}. The other direction is obtained in exactly same way as in the standard proof of the filter extension property from the assumption of strong compactness: 

Let $\mathcal{F}$ be a $\kappa$-complete filter and let $M$ be sufficiently nice transitive model such that $\mathcal{F} \in M$, and $2^\lambda, \power(\lambda)\subseteq M$.  In $N$, $j\image \mathcal{F}$ is covered by a set of size $<j(\kappa)$, $S$. Without loss of generality, $S \subseteq j(\mathcal{F})$. Therefore, $\bigcap S \neq \emptyset$. Let us pick any $t\in\bigcap S$ and let us define $\mathcal{U} = \{X \subseteq \lambda \mid t\in j(X)\}$. $\mathcal{U}$ is clearly $\kappa$-complete ultrafilter that extends $\mathcal{F}$.
\end{proof}

Using the elementary embedding from Lemma \ref{lemma: elementary embedding}, one can obtain from the $\lambda$-filter extension property many of the standard results of partial strongly compact cardinal (see \cite{Solovay1974}).
\begin{corollary}\label{cor: reflections}
Let $\kappa \leq \lambda$ and let us assume that the $\lambda$-filter extension property holds and $\lambda^{<\kappa} = \lambda$. Let $\mu \in [\kappa, 2^\lambda]$ be a regular cardinal.
\begin{enumerate}
\item Every collection of $<\kappa$ many stationary subsets of $S^{\mu}_{<\kappa}$ has a common reflection point. 
\item $\square(\mu, <\kappa)$ fails.
\item If $G$ is a graph of size $\mu$, and $\rho < \kappa$ is a cardinal such that every subgraph of size $<\kappa$ of $G$ has chromatic color at most $\rho$, then $G$ has chromatic color at most $\rho$.
\end{enumerate}
\end{corollary}
Let us focus in the case of $\lambda = \kappa$ (the case of $\kappa$-compactness). 
In \cite{Mitchell79}, Mitchell asked what is the consistency strength of the existence of an uncountable  cardinal $\kappa$ that has the $\kappa$-filter extension property. Mitchell conjectured that the consistency strength of this property is in the realm of $o(\kappa)=\kappa^{++}$. In the paper \cite{Gitik2018}, Gitik coined the term \emph{$\kappa$-compact} for this property and investigate many aspects of it. In particular, he showed that counter-intuitively, $\kappa$-compactness is much stronger that $o(\kappa) = \kappa^{++}$. Indeed, he showed that if $\kappa$ is $\kappa$-compact then there is an inner model with a Woodin cardinal. In Corollary \ref{cor:consistency strength} ahead we will improve Gitik's lower bound. 

Note that any of the assertions in Corollary \ref{cor: reflections} implies the failure of $\square(\mu)$ for any regular cardinal $\mu$ in the interval $[\kappa, 2^\kappa]$, and in particular the failure of $\square(\kappa^+)$, if $\kappa$ is $\kappa$-compact. This will be used ahead in Corollary \ref{cor:consistency strength} in order to derive a significant consistency strength from the assumption of $\kappa$-compactness:
\begin{corollary}\label{cor:consistency strength}
If $\kappa$ is $\kappa$-compact then there is an inner model with a proper class of strong cardinals and a proper class of Woodin cardinals.
\end{corollary}
\begin{proof}
By Corollary \ref{cor: reflections}, if $\kappa$ is $\kappa$-compact, then $\square(\kappa)$ and $\square(\kappa^{+})$ both fail. By \cite{Jensen2009StackingMice}, the failure of $\square(\kappa)$ together with $\square_\kappa$ for a countably closed cardinal $\kappa\geq \aleph_3$ implies the existence of an inner model with a proper class of Woodin cardinals and a proper class of strong cardinals.
\end{proof}

One can use the characterization in Corollary \ref{cor: filter extension to elementary embedding} in order to obtain consistency results from the assumption of $\kappa$-compactness which are usually obtained by using $\kappa^{+}$-strongly compact cardinals. 

For example, let us consider a bounded variant of Rado's conjecture. Let us say that a tree $T$ is \emph{$\mu$-special} if it is a union of at most $\mu$ many antichains. Clearly, if a tree is $\mu$-special then its height is at most $\mu^+$, and any tree with height $<\mu^+$ is $\mu$-special. Thus, in this context we are interested only in trees of height $\mu^+$.  

Let us use Fuchino's notations from \cite{Fuchino2017}: $RC(\mu, <\kappa, \lambda)$ is the assertion that a tree $T$ (of height $\mu^+$), with $|T| = \lambda$ is $\mu$-special if and only if every sub-tree $T' \subseteq T$ of size $<\kappa$ is $\mu$-special. 

By Corollary \ref{cor: reflections}, if $\kappa$ is $\kappa$-compact then $RC(\rho,<\kappa,\kappa^+)$ holds for all $\rho < \kappa$. By applying exactly the same proof as in \cite{Todorcevic1983}, but replacing the supercompact (or strongly compact) embedding with the elementary embeddings from Lemma \ref{lemma: elementary embedding}, one can force $\kappa$ to be $\omega_2$ and obtain the consistency of local version of Rado's conjecture from $\kappa$-compactness. Note that the elementary embedding which is used in the proof depends on the non-special tree which we prove that has a non-special small sub-tree.
\begin{remark}
Let $\kappa$ be $\kappa$-compact. The L\'{e}vy collapse $\Col(\omega_1,<\kappa)$, forces that $RC(\aleph_0, <\aleph_2, \aleph_3)$ holds.
\end{remark}
By the arguments of \cite[Theorem 10]{Todorcevic1993}, $RC(\mu, <\kappa, \kappa^+)$ already implies the failure of $\square(\kappa^+)$. It is unclear how weaker square principles are affected by Rado's conjecture. In \cite{PerezWu2017}, P\'{e}rez and Wu show that Rado's conjecture implies the failure of $\square(\lambda, \omega)$ for all regular $\lambda > \omega_1$. Nevertheless, it is unclear whether the conjunction $\bigwedge_{\mu < \kappa} RC(\mu, <\kappa, \kappa^+)$ implies the failure of $\square(\kappa^+,\omega_1)$.

In \cite{Zhang2017}, Zhang investigates a weaker variant of Rado's conjecture, $RC^b$. This principle is still sufficiently strong in order to imply stationary reflection for subsets of cofinality $\omega$ and the failure of $\square(\mu)$ for all regular $\mu > \omega_1$. Zhang shows that this version does not imply simultaneous stationary reflection for pairs for stationary subsets of $S^{\omega_2}_\omega$. It is still open whether Rado's conjecture implies any instance of simultaneous stationary reflection.

\section{Consistency strength}\label{section:upper bounds}
Let us start with a definition of large cardinal notion in the realm of supercompactness that plays an important role in the analysis of the filter extension property. 
\begin{definition}
Let $\kappa \leq \lambda$ be cardinals. $\kappa$ is $\lambda$-$\Pi^1_1$-subcompact if for every $A\subseteq H(\lambda)$ and every $\Pi^1_1$-statement $\Phi$, if $\langle H(\lambda), \in, A\rangle \models \Phi$ then there is $\rho < \kappa$, $\bar{\lambda} < \kappa$ and $B\subseteq H(\bar\lambda)$ such that:
\[\langle H(\bar\lambda),\in B\rangle \models \Phi\] 
and there is an elementary embedding:
\[j\colon \langle H(\bar\lambda), \in, B\rangle \to \langle H(\lambda), \in, A\rangle\]
with critical point $\rho$, such that $j(\rho) = \kappa$. 
\end{definition}
For the case $\kappa=\lambda$, since $\kappa\notin H(\lambda)$, we require that $j = id$. Thus, a cardinal $\kappa$ is $\kappa$-$\Pi^1_1$-subcompact if and only if it is weakly compact.

The important case $\lambda = \kappa^+$ was introduced by Neeman and Steel in \cite[Section 1]{NeemanSteelSubcompact}. In their paper this type of cardinal is called \emph{$\Pi^2_1$-subcompact}. 

The following theorem improves slightly \cite[Theorem 1.1]{Gitik2018}.
\begin{theorem}
Let $\kappa \leq \lambda$ be $\lambda^+$-$\Pi^1_1$-subcompact cardinal. Then $\kappa$ has the $\lambda$-filter extension property. In particular, any $\kappa^+$-$\Pi^1_1$-subcompact cardinal is $\kappa$-compact.
\end{theorem}
\begin{proof}
Let $\mathcal{F}$ be a $\kappa$-complete filter on $\lambda$. Assume that there is no $\kappa$-complete ultrafilter $\mathcal{U}$ extending $\mathcal{F}$. This is a $\Pi^1_1$-statement in $H(\lambda^+)$ (with parameter $\mathcal{F}$). Let us denote this statement by $\Phi$.

Thus, by the $\lambda^+$-$\Pi^1_1$-subcompactness, there is $\rho < \bar{\lambda} < \kappa$ and an elementary embedding:
\[j \colon \langle H(\bar{\lambda}^{+}), \in \bar{\mathcal{F}}\rangle \to \langle H(\lambda^{+}), \in, \mathcal{F}\rangle\] 
such that there is no $\rho$-complete ultrafilter on $\power(\bar{\lambda})$ extending $\bar{\mathcal{F}}$. 

Let us look at $j\image \bar{\mathcal{F}}$. This is a subset of $\mathcal{F}$ of size $2^{\bar{\lambda}} < \kappa$. Thus, there is an element $s \in \bigcap_{A\in \bar{\mathcal{F}}} j(A)$. Let $\bar{\mathcal{U}}$ be the measure generated by $s$, namely, $X \in \bar{\mathcal{U}}$ iff $s \in j(X)$. 

This measure is $\rho$-complete, as for every sequence of $<\rho$ sets $\langle X_i \mid i < \eta < \rho\rangle$, this sequence is a member of $H(\bar{\lambda}^+)$ and thus one can apply $j$ on it and get:
\[s \in \bigcap_{i < \eta} j(X_i) = j\left(\bigcap_{i < \eta} X_i\right).\]
We conclude that $\bigcap_{i<\eta} X_i \in \bar{\mathcal{U}}$, so $\bar{\mathcal{U}}$ is a $\rho$-complete ultrafilter which extends $\bar{\mathcal{F}}$, which is a contradiction to the assumption that $H(\rho^+)$ satisfies $\Phi$. 
\end{proof}

By standard reflection arguments (originated from \cite{Magidor1971}), if $\kappa$ is $\lambda$-supercompact then it is $\lambda$-$\Pi^1_1$-subcompact. Moreover, if $\lambda = \kappa^{+\alpha}$ for $\alpha < \kappa$, then the set of all $\rho < \kappa$ such that $\rho$ is $\rho^{+\alpha}$-$\Pi^1_1$-subcompact belongs to the normal measure on $\kappa$ which is derived from any $\lambda$-supercompact embedding. 

Let $V$ be a model of $\GCH$ and level-by-level equivalence of strong compactness and supercompactness, as in \cite{ApterShelah97}, and let us assume that there is a cardinal $\kappa$ which is $\kappa^+$-strongly compact in $V$. Then in $V$ there are many cardinals $\rho < \kappa$ in which are $\rho$-compact and not $\rho^+$-strongly compact.  

By \cite[Theorem 4.6]{NeemanSteelSubcompact}, if there is a weakly iterable premouse such that $\left(\kappa^{+}\right)^V = \left(\delta^{+}\right)^{\mathcal{Q}}$ for some cardinal $\delta$ and $\square(\kappa^+)$ fails in $V$, then $\delta$ is $\delta^+$-$\Pi^1_1$-subcompact in $\mathcal{Q}$. In particular, if $\mathcal{Q}$ is a weakly iterable premouse such that $\mathcal{Q}\models$``$\kappa$ is $\kappa$-compact'' then $\mathcal{Q}\models \kappa$ is $\kappa^+$-$\Pi^1_1$-subcompact.
  
Thus, it is natural to conjecture: 
\begin{conjecture}
The existence of $\kappa$ which is $\kappa$-compact is equiconsistent with the existence of a cardinal $\delta$ which is $\delta^+$-$\Pi^1_1$-subcompact.
\end{conjecture}

\section{Acknowledgments}
I would like to thank Moti Gitik and Menachem Magidor for helpful discussions. I would like to thank Specer Unger and Assaf Rinot for reviewing early drafts of the paper and making many helpful suggestions.  

I Would like to thank Stevo \Todorcevic for his helpful suggestions.

Finally, I would like to thank the anonymous referee for the comments and helpful suggestions on early draft of the paper.
\providecommand{\bysame}{\leavevmode\hbox to3em{\hrulefill}\thinspace}
\providecommand{\MR}{\relax\ifhmode\unskip\space\fi MR }
\providecommand{\MRhref}[2]{%
  \href{http://www.ams.org/mathscinet-getitem?mr=#1}{#2}
}
\providecommand{\href}[2]{#2}

\end{document}